\documentclass[12pt]{amsart}

\usepackage{tikz}
\usetikzlibrary{shapes.geometric}

\usepackage{amssymb}
\usepackage{amsfonts}
\usepackage{latexsym}
\usepackage{stmaryrd}
\usepackage{yhmath}
\usepackage{setspace}

\theoremstyle{plain}
\newtheorem{theorem}{Theorem}
\newtheorem{corollary}[theorem]{Corollary}
\newtheorem{lemma}[theorem]{Lemma}

\newtheorem{definition}[theorem]{Definition}

\numberwithin{equation}{section}
\numberwithin{theorem}{section}

\newcommand{\norm}{\mathbin{\lhd}}

\newcommand{\cov}{%
    \mathbin{\mbox{%
               \makebox[.8ex][l]{%
                  \raisebox{.725ex}{%
                     \rule{1ex}{.12ex}}}%
               \Large $\Yleft$}}}

\renewcommand{\cov}{\Yleft}

\newcommand{\meet}{\wedge}
\newcommand{\join}{\vee}
\let\hat=\widehat
\newcommand{\N}{\ensuremath{\mathbb{N}}}

\newcommand{\C}{\mathbb{C}}
\newcommand{\supstyle}[1]{{\sf #1}}
\newcommand{\AAA}{\supstyle{A}}

\newcommand{\CCC}{\supstyle{C}}
\newcommand{\DDD}{\supstyle{D}}
\newcommand{\EEE}{\supstyle{E}}
\newcommand{\aaa}{\supstyle{a}}
\newcommand{\bbb}{\supstyle{b}}
\newcommand{\ccc}{\supstyle{c}}
\newcommand{\ddd}{\supstyle{d}}
\newcommand{\eee}{\supstyle{e}}
\newcommand{\fff}{\supstyle{f}}
\newcommand{\rrr}{\supstyle{r}}
\newcommand{\sss}{\supstyle{s}}
\newcommand{\uuu}{\supstyle{u}}
\newcommand{\vvv}{\supstyle{v}}
\newcommand{\www}{\supstyle{w}}
\newcommand{\xxx}{\supstyle{x}}
\newcommand{\yyy}{\supstyle{y}}
\newcommand{\zzz}{\supstyle{z}}

\newcommand{\KK}{\mathcal{K}}
\newcommand{\LL}{\mathcal{L}}
\newcommand{\xx}{\mathcal{X}} \let\XX=\xx
 \let\YY=\yy
 
\newcommand{\A}[1]{\AAA_{#1}}
\newcommand{\defnstyle}[1]{{\bfseries\em #1}}
\newcommand{\gen}[1]{\left\langle #1 \right\rangle}

\newcommand{\Max}[1]{\ensuremath{\supstyle{M}_{#1}}}
\newcommand{\Min}[2][]{\supstyle{m}_{#2}^{#1}}
\newcommand{\xinv}[1]{\supstyle{I}_{#1}}

\newcommand{\by}{\times}
\newcommand{\idp}{\mathbin{\dot{\times}}}
\renewcommand{\phi}{\varphi}
\DeclareMathOperator{\Part}{Part}
\DeclareMathOperator{\Sup}{Sup}
\DeclareMathOperator{\Aut}{Aut}
\DeclareMathOperator{\Irr}{Irr}

\newcommand{\sseq}{\subseteq}
\newcommand{\pleq}{\leq}

\def\marginpar#1{}

\begin{document}

\title[upper and lower semimodularity conditions]
      {Upper and lower semimodularity of the supercharacter theory lattices of cyclic groups}
\date{\today}

\author[Benidt]{Samuel~G. Benidt}
\address[Samuel~G. Benidt]{Concordia College, Department of Mathematics and Computer Science, Moorhead, MN 56562 USA}
\email{sgbenidt@gmail.com}
\author[Hall]{William~R.S. Hall}
\address[William~R.S. Hall]{Stony Brook University, Department of Applied Mathematics and Statistics, Stony Brook, NY 11794 USA}
\email{whall@ams.sunysb.edu}
\author[Hendrickson]{Anders~O.F. Hendrickson}
\address[Anders~O.F. Hendrickson]{Concordia College, Department of Mathematics and Computer Science, Moorhead, MN 56562 USA}
\email{ahendric@cord.edu}
%\subjclass

\begin{abstract}
We consider the lattice of supercharacter theories, in the sense of Diaconis and Isaacs \cite{diaconis_isaacs},
of the cyclic group of order $n$.
We find necessary and sufficient conditions on $n$ for that lattice to be upper or lower semimodular.
\end{abstract}

\maketitle

\thispagestyle{empty}

\def\Cyc#1{C_{#1}}

%%%%%%%%%%%%%%%%
% INTRODUCTION %
%%%%%%%%%%%%%%%%

Diaconis and Isaacs defined supercharacter theories 
as generalizations of ordinary character theory
which use certain (generally reducible) characters in place of irreducible characters,
and a coarser partition of the group in place of the partition into conjugacy classes \cite{diaconis_isaacs}.
Much attention has been paid to a certain supercharacter theory of algebra groups 
which is useful in random walk problems \cite{arias2004},
and Aguiar et al.~have discovered striking connections with the Hopf algebra of symmetric functions of noncommuting variables \cite{aimpaper}.
Supercharacter theories of cyclic groups (indeed, of all abelian groups) 
are in bijective correspondence with the groups' Schur rings \cite{ssr},
and in that language they have long been of great interest to
algebraic combinatorists studying circulant graphs \cite{muzychuk_ponomarenko}.
More recently, C. Fowler, S. Garcia, and G. Karaali
have shown that several Ramanujan sum identities
can be derived easily using the machinery of supercharacter theories of cyclic groups \cite{fowler_garcia_karaali}.
For these reasons, the set of supercharacter theories of the cyclic group $\Cyc{n}$ is worth studying.

The set of all supercharacter theories of a given group forms a partially ordered set, and indeed a lattice.
In this paper, we find necessary and sufficient conditions on $n$
for the supercharacter theory lattice of $\Cyc{n}$
to be upper or lower semimodular.

Section \ref{sect_background}
reviews the necessary definitions and notation for supercharacter theories and for lattices.
Sections \ref{sect_starprod} and \ref{sect_directproducts}
will investigate lattice-theoretic properties
of $*$-products and of
direct product theories of $\Cyc{pq}$, respectively.
Using the results of those sections,
we prove in section \ref{sect_upper} that $\Sup(\Cyc{n})$ is upper semimodular if and only if $n$ is prime or four;
then in the final section %\ref{sect_lower} 
we prove that $\Sup(\Cyc{n})$ is lower semimodular
precisely when $n$ is prime, four, or of the form $pq$.

\section{Background}\label{sect_background}

All groups in this paper are finite.
Diaconis and Isaacs defined supercharacter theories as follows:
\begin{definition}[{\cite[section 2]{diaconis_isaacs}}]
Let $G$ be a finite group, let $\KK$ be a partition of $G$, and let $\XX$ be a
partition of the set of irreducible characters $\Irr(G)$. 
Suppose that for every part $X \in \XX$ there exists a character
$\chi_X$ whose irreducible constituents lie in $X$, such that the following three conditions hold:
\begin{enumerate}
  \item $|\XX| = |\KK|$.
  \item Each of the characters $\chi_X$ is constant on every part of $\KK$.
  \item $\{1\}\in\KK$. %Every irreducible character is a constituent of some $\chi_X$.
\end{enumerate}
Then we say the ordered pair $(\xx,\KK)$ is a \defnstyle{supercharacter theory} of $G$,
and we write $\Sup(G)$ for the set of all supercharacter theories.
The parts of $\KK$ are called the \defnstyle{superclasses}, and the characters $\chi_X$ are called the \defnstyle{supercharacters}.
\end{definition}

If $g\in G$ and $\ccc\in\Sup(G)$,
the superclass of $g$ in $\ccc$ will be denoted $[g]_\ccc$,
or simply $[g]$ if the supercharacter theory is unambiguous.
For every subset $K\sseq G$, let $\hat{K}$ denote the sum $\sum_{g\in K} g$ in the group algebra $\C[G]$.
The following fact is exceedingly useful:

\begin{lemma}[{\cite[Corollary 2.3]{diaconis_isaacs}}]\label{superclassproducts}
  Suppose that $K$ and $L$ are superclasses in some supercharacter theory for a group $G$.
  Then $\hat{K}\hat{L}$ is a nonnegative integer linear combination of superclass sums in the group algebra $\C[G]$.
\end{lemma}

%%%In fact, if $(\XX,\KK)\in\Sup(G)$, then each of the partitions $\XX$ and $\KK$ determines the other \cite[??]{diaconis_isaacs}.
%%%Moreover, we can determine whether a partition $\KK$ contains the superclasses of a supercharacter theory
%%%merely by computations in the group algebra, without resorting to characters.
%%%\begin{lemma}[{\cite[ref?]{diaconis_isaacs}}]
%%%Let $G$ be a finite group and let $\KK$ be a partition of $G$ 
%%%Then there exists a partition $\XX$ of $\Irr(G)$ such that $(\XX,\KK)\in\Sup(G)$ if and only if
%%%\begin{enumerate}
%%%  \item $\{1\}\in\KK$,
%%%  \item each part of $\KK$ is a union of conjugacy classes, and
%%%  \item The subspace $\operatorname{span}\{\hat{K}: K\in\KK\}$
%%%        of $\C[G]$ is closed under multiplication.
%%%\end{enumerate}
%%%\end{lemma}
%%%In other words, the third condition guarantees that for all superclasses $K$ and $L$, 
%%%the product $\hat{K}\hat{L}$ is a linear combination of superclass sums.

%%  For ease of notation, 
%%  we will often list a partition by using
%%  slashes to denote where parts begin and end, 
%%  while commas set off the individual elements in each part.
%%  For example, we will write the partition
%%  $\{\{a_1\}, \{a_2, a_3\}, \{a_4\}\}$
%%  more concisely as 
%%  $\{a_1/ a_2,a_3 / a_4\}.$

Supercharacter theories arise in several natural ways.
First, 
let $A$ be a group acting on $G$ by automorphisms.
Then $A$ also acts on $\Irr(G)$,
and the orbits of this action 
yield a supercharacter theory of $G$ \cite[section 1]{diaconis_isaacs},
which we say \defnstyle{comes from automorphisms}
or \defnstyle{comes from $A$}.
(If $G$ is abelian, then the superclass of $g\in G$ is simply its orbit $g^A$.)
%Elements whose conjugacy classes lie in the same orbit are joined into the same part.
%We let $\Aut(G)$ denote the set of all automorphisms of $G$, and each supercharacter theory from automorphisms comes from a subgroup of $\Aut(G)$.
In particular, $\A{G}$ will denote the supercharacter theory which comes from the full automorphism group $\Aut(G)$.

Another construction is the $*$-product,
which builds a supercharacter theory of a group $G$
out of supercharacter theories of a normal subgroup $N$ and of the quotient $G/N$.
Let $\CCC=(\XX,\KK)\in\Sup(N)$ be invariant under conjugation by $G$, and let $\DDD=(\YY,\LL)\in\Sup(G/N)$.
The partition $\LL$ of $G/N$ yields a partition of $G$ into unions of $N$-cosets,
one part of which is $N$.
By replacing that part with the partition $\KK$ of $N$, we obtain a partition of $G$
into the superclasses of a supercharacter theory $\CCC*\DDD\in\Sup(G)$
defined in \cite{ssr}, which is called a \defnstyle{$*$-product over $N$}.
It is a special case of a more general construction called a \defnstyle{$\triangle$-product};
we refer the reader to \cite{ssr} for details.

A third construction is possible if $G$ is the direct product of subgroups $A$ and $B$.
If $\aaa\in\Sup(A)$ and $\bbb\in\Sup(B)$,
then there exists a \defnstyle{direct product} $\aaa\by\bbb\in\Sup(G)$
such that for all $x\in A$ and $y\in B$, 
the superclass 
$$[xy]_{\aaa\by\bbb} = [x]_\aaa [y]_\bbb = \{x'y': x'\in [x]_\aaa, y'\in [y]_\bbb\}.$$
See \cite{ssr} for details.

In addition to the preceding constructions, 
every group $G$ also has
its maximal supercharacter theory $\Max{G}$ with superclasses $\{\{1\},G-\{1\}\}$
and its minimal supercharacter theory $\Min{G}$,
whose superclasses are simply the conjugacy classes of $G$.
For all $\EEE \in \Sup(G)$,
we say a subgroup $N\leq G$ is \defnstyle{$\EEE$-normal} if $N$ is a union of superclasses of $\EEE$.

\medskip

The subject of this paper is lattice-theoretic properties of the set $\Sup(\Cyc{n})$.
Recall that a lattice is a partially ordered set in which any two elements 
$a$ and $b$ have a unique least upper bound or \defnstyle{join} $a\join b$ and greatest lower bound or \defnstyle{meet} $a\meet b$.
We say $b$ \defnstyle{covers} $a$ and write $a\cov b$
if $a<b$ and no element $c$ satisfies $a<c<b$.
The set $\Part(S)$ of all set-partitions of a set $S$ is one natural example of a lattice.
Viewed as a subset of $\Part(G)$, the set $\Sup(G)$ contains
a maximal element $\Max{G}$ and minimal element $\Min{G}$,
so it is itself a lattice;
this partial order is also compatible with the partitions of characters \cite[Corollary 3.4]{ssr}.
In particular, 
the join of two supercharacter theories is simply their join as set-partitions,
although their meet in $\Sup(G)$ is generally not their meet as elements of $\Part(G)$ \cite[Proposition 3.3]{ssr}.

%A bijection $\alpha$ between lattices is called a \defnstyle{lattice isomorphism} 
%if it preserves the partial order.  In particular, 
%$\alpha(x\join y)= \alpha(x) \join \alpha(y)$ and $\alpha(x \meet y)= \alpha(x) \meet \alpha(y)$ for all $x,y\in\XX$.

The general definition of semimodularity is somewhat involved \cite[p.~142]{szasz},
but for lattices of finite length such as $\Sup(G)$,
Birkhoff's original formulation is equivalent: %\cite[Definition 5.3]{abbott} 
\begin{definition}[{\cite[{\S8}]{birkhoff}}]
A lattice $L$ of finite length is said to be \defnstyle{upper semimodular} if the following condition is satisfied
for all $a,b\in L$:
$$ \mbox{~if~}a \meet b \cov a,b, \mbox{~then~} a,b \cov a \join b.$$
In other words, if $a$ and $b$ cover their meet, then they are both covered by their join.
Likewise, $L$ is \defnstyle{lower semimodular} if for all $a, b \in L$,
$$ \text{ if } a,b \cov a \join b \text{ then } a \meet b \cov a,b;$$
that is, if $a$ and $b$ are covered by their join, then they cover their meet.
A lattice is \defnstyle{modular} if it is both upper and lower semimodular.
\end{definition}

\medskip
We shall make frequent use of the following theorem of Leung and Man, % \cite[Theorem 3.7]{leungman}, 
as rephrased in \cite[Theorem 8.2]{ssr}:
\begin{theorem}[{\cite[Theorem 3.7]{leung_man1996}}]\label{thm:classification}
  Let $G$ be a finite cyclic group and let $\CCC$ be a non-maximal supercharacter theory of $G$.
  Then at least one of the following is true:
  \begin{enumerate}
    \item $\CCC$ comes from automorphisms
    \item $\CCC$ is a nontrivial direct product
    \item $\CCC$ is a nontrivial $\triangle$-product
  \end{enumerate}
\end{theorem}

This theorem implies several special cases, which can of course be proved more easily.
If $G$ is cyclic of prime order, 
   then all of its supercharacter theories come from automorphisms.
If $G$ is cyclic of order $pq$ or $p^2$ and $\CCC\in\Sup(G)$, 
   then $\CCC$ either 
      is maximal, 
      comes from automorphisms, 
      or is a nontrivial $*$-product.

\medskip
We close this section with the initial result that 
the lattice $\Sup(\Cyc{p})$ is modular.

\begin{lemma}\label{Aut(G)}
Let $G$ be a cyclic group. Let $C$ be the lattice of subgroups of $\Aut(G)$,
%%%%%%ORR%%%%%
%\A = \{supercharacter theories from automorphisms\}
%$\alpha: C \mapsto A$ is an isomorphism.
and let $L$ be the set of supercharacter theories of $G$ that come from automorphisms.
For each subgroup $H$ of $\Aut(G)$, 
let $\alpha(H)\in L$ be the supercharacter theory of $G$ that comes from $A$. 
Then $\alpha$ is a lattice isomorphism from $C$ to $L$.
\end{lemma}

\begin{proof}
%We first show that $\alpha$ is one-to-one.
Certainly $\alpha$ is surjective,
so suppose $\alpha(A)=\alpha(B)$
for some $A,B \leq\Aut(G)$.
Let $g$ be a generator of $G$.
Then the the superclass of $\alpha(A)$ containing $g$ is equal to the superclass of $\alpha(B)$ containing $g$,
so $g^A = g^B$.
As every automorphism of a cyclic group is determined by where it sends a generator,
it follows that $A=B$, so $\alpha$ is injective.

It remains to show that $\alpha$ preserves the partial order.
%$$A \pleq B \text{ if and only if } \alpha(A) \pleq \alpha(B).$$
If $\alpha(A) \pleq \alpha(B)$, then $g^A\sseq g^B$ and hence $A \leq B$.  
Conversely, if $A \leq B$, then $x^A\sseq x^B$ for each $x\in G$,
%so every superclass of $\alpha(A)$ is contained in a superclass of $\alpha(B)$, 
and thus $\alpha(A) \pleq \alpha(B)$.
We conclude that $\alpha$ is a lattice isomorphism.
\end{proof}

\begin{corollary}\label{Cp_modular}
  Let $G$ be cyclic of order $p$.
  Then the lattice $\Sup(G)$ is modular,
  and hence both upper and lower semimodular.
\end{corollary}
\begin{proof}
  Because every supercharacter theory of $G$ comes from automorphisms,
  Lemma \ref{Aut(G)} implies that 
  $\Sup(G)$ is isomorphic to the lattice of subgroups of $\Aut(G)$.
  Now the lattice of all normal subgroups of a group is modular, as is well-known 
  (e.g., \cite[Theorem 11]{birkhoff});
  %(e.g., \cite[p.~6]{suzuki_lattices});
  since $\Aut(G)$ is abelian, 
  it follows that its subgroup lattice is modular
  and so is $\Sup(G)$.
\end{proof}

\section{Sublattices of $*$-Products}\label{sect_starprod}

Let $G$ be a finite group, and let $N\norm G$.
Let $\Sup_G(N)$ denote the set of supercharacter theories of $N$ 
that are invariant under conjugation by $G$.
If $G$ is abelian then $\Sup_G(N)=\Sup(N)$, of course,
but we prefer to state the following lemmas in full generality.
In this section we shall investigate lattice-theoretic properties of the set 
\begin{equation*}
  \Sup_G(N)*\Sup(G/N)=\{\xxx*\yyy: \xxx\in\Sup_G(N),~\yyy\in\Sup(G/N)\}
\end{equation*}
of all $*$-products over $N$.
Recall from \cite{ssr} that $\eee\in\Sup(G)$ is a $*$-product over $N$ 
%if and only if $N$ is $\eee$-normal and every superclass of $\eee$ outside $N$ is a union of $N$-cosets.
%This can also be expressed by saying $\eee$ is a $*$-product over $N$ 
if and only if $\Min[G]{N}*\Min{G/N}\pleq \eee \pleq \Max{N}*\Max{G/N}$,
where $\Min[G]{N}$ denotes the minimal $G$-invariant supercharacter theory of $N$.
It follows that for $\aaa,\bbb,\ccc\in\Sup(G)$ with $\aaa\leq \bbb\leq \ccc$,
if $\aaa$ and $\ccc$ are $*$-products over $N$ then so is $\bbb$;
in other words, the set $\Sup_G(N)*\Sup(G/N)$ is \defnstyle{convex}.

The following lemma is clear from the definition of the $*$-product, so we omit its proof.
\begin{lemma}\label{lemprec}
Let $G$ be a group, and let $N\norm G$. Let $\xxx_1, \xxx_2 \in \Sup_G(N)$ and $\yyy_1, \yyy_2 \in \Sup(G/N)$. 
Then 
 $\xxx_1*\yyy_1 \pleq \xxx_2*\yyy_2$
 if and only if 
 $\xxx_1 \pleq \xxx_2$ and $\yyy_1\pleq \yyy_2$. 
Furthermore, 
 $\xxx_1*\yyy_1 = \xxx_2*\yyy_2$
 if and only if $\xxx_1=\xxx_2$ and $\yyy_1=\yyy_2$.
\end{lemma}

Recall that a subset of a lattice is called a \defnstyle{sublattice}
if it is closed under joins and meets.
%and a sublattice $X$ of $L$ is called \defnstyle{convex} if 
%whenever $a,b,c\in L$ and $a,c\in X$, if $a\pleq b\pleq c$ then $b\in X$.
Also recall that if $X$ and $Y$ are lattices,
then their direct union $X\otimes Y$ 
is the Cartesian product $X\times Y$
under the partial order defined by letting $(x_1,y_1) \pleq (x_2,y_2)$ if and only if $x_1 \pleq x_2$ and $y_1 \pleq y_2$.

\begin{lemma}\label{lemClosure}
Let $G$ be a finite group, and let $N\norm G$.
Let $A$ be a convex sublattice of $\Sup_G(N)$, and $B$ be a convex sublattice of $\Sup(G/N)$. 
Then 
$$A*B=\{\aaa*\bbb : \aaa \in A, \bbb\in B\}$$ 
is a convex sublattice of $\Sup(G)$ 
which is lattice isomorphic to $A\otimes B$.
\end{lemma}

\begin{proof}
  Let $\xxx_1=\aaa_1*\bbb_1$ and $\xxx_2=\aaa_2*\bbb_2$ be arbitrary elements of $A*B$.
  %Let $\xxx_1, \xxx_2\in A*B$.
  Certainly $\Min[G]{N}*\Min{G/N}\pleq \xxx_i\pleq \Max{N}*\Max{G/N}$ 
  for $i=1,2$,
  so $\xxx_1\meet\xxx_2$ is also a $*$-product over $N$.
  It follows then from Lemma \ref{lemprec} 
  that
  $$(\aaa_1*\bbb_1) \meet (\aaa_2*\bbb_2) = (\aaa_1\meet \aaa_2) * (\bbb_1 \meet \bbb_2),$$
  which lies in $A*B$ because $A$ and $B$ are sublattices.
  Thus $A*B$ is closed under meets; a similar argument shows it to be closed under joins as well,
  and thus it is a sublattice.
  
  To prove convexity, suppose $\aaa_1*\bbb_1 \pleq \ccc \pleq \aaa_2*\bbb_2$
  for some $\ccc\in\Sup(G)$.
  Then $\ccc$ is a $*$-product over $N$;
  writing $\ccc=\uuu*\vvv$ for some $\uuu\in\Sup_G(N)$ and $\vvv\in\Sup(G/N)$,
  we have
  $\aaa_1*\bbb_1 \leq \uuu*\vvv \leq \aaa_2*\bbb_2$,
  so $\aaa_1\pleq \uuu\pleq\aaa_2$ and $\bbb_1\pleq\vvv\pleq\bbb_2$
  by Lemma \ref{lemprec}.
  Then $\uuu\in A$ and $\vvv\in B$ by the convexity of those lattices,
  so $\ccc\in A*B$ as desired.
  
  Finally, the mapping from $A \otimes B$ to $A*B$ 
  which maps $(\aaa,\bbb) \mapsto \aaa*\bbb$
  is clearly surjective and is injective by Lemma \ref{lemprec};
  moreover,
  by definition $(\aaa_1,\bbb_1) \pleq (\aaa_2,\bbb_2)$ in $A \otimes B$
  if and only if $\aaa_1 \pleq \aaa_2$ and $\bbb_1 \pleq \bbb_2$,
  which by Lemma \ref{lemprec} is true if and only if
  $\aaa_1*\bbb_1 \pleq \aaa_2*\bbb_2$.
  Thus this map is a lattice isomorphism from $A\otimes B$ to $A*B$.
\end{proof}

The following corollary, which provides infinite families of 
groups for which $\Sup(G)$ is not upper (or lower) semimodular,
will be essential to our main theorems in sections 4 and 5.
\begin{corollary}\label{inheritance}
Let $G$ be an abelian group, and let $N\leq G$.
\begin{enumerate}
\item If $\Sup(G)$ is upper semimodular, then so is $\Sup(N)$.
\item If $\Sup(G)$ is lower semimodular, then so is $\Sup(N)$.
\end{enumerate}
In particular, if $\Cyc{n}$ is not upper (or lower) semimodular for some $n\in\N$,
then $\Cyc{nk}$ also fails to be upper (or lower) semimodular for all $k\in\N$.
\end{corollary}

\begin{proof}
%  It is known that every convex sublattice of an upper semimodular lattice is upper semimodular \cite[Theorem 120]{donnellan}. 
%  If $\Sup(G)$ is upper semimodular, 
%  then its convex sublattice $B=\{\xxx*\Max{G/N} | \xxx \in \Sup(N) \}$ is upper semimodular, 
%  and hence $\Sup(N)$ is upper semimodular since $\Sup(N)$ is isomorphic to $B$ by Lemma \ref{lemClosure}.
%  
  Suppose $\Sup(G)$ is upper semimodular.
  Then $\Sup(N)$ is lattice-isomorphic to $\Sup(N)\otimes\{\Max{G/N}\}$ and thus to $\Sup(N)*\{\Max{G/N}\}$ by Lemma \ref{lemClosure},
  which is a convex sublattice of $\Sup(G)$.
  Because every convex sublattice of an upper semimodular lattice is also upper semimodular \cite[Theorem 120]{donnellan},
  the first statement follows.
  The second statement follows from a similar argument,
  and the final claim holds since $\Cyc{n}$ embeds in $\Cyc{nk}$.
\end{proof}

Because $\Sup_G(N)*\Sup(G/N)$ is lattice-isomorphic to the direct union of lattices $\Sup_G(N)$ and $\Sup(G/N)$
by Lemma \ref{lemClosure},
the covering relationships are particularly nice.
%We next investigate covering within the sublattice $\Sup(N)*\Sup(G/N)$.
%\marginpar{\tiny\raggedright Do more by just isomorphism.}

\begin{lemma}\label{lemcover}
Let $G$ be a group and let $N\lhd G$. Let $\xxx_1,\xxx_2 \in \Sup_G(N)$ and $\yyy_1,\yyy_2 \in \Sup(G/N)$. 
Then $\xxx_1*\yyy_1 \cov \xxx_2*\yyy_2$ if and only if
either $\xxx_1=\xxx_2$ and $\yyy_1 \cov \yyy_2$, 
or $\yyy_1=\yyy_2$ and $\xxx_1 \cov \xxx_2$.
\end{lemma}
\begin{proof}
  This follows immediately from the corresponding statement for direct unions of lattices.\marginpar{ref?}
\end{proof}
%\begin{proof}
%First we claim that $\X_1 \cov \X_2$ if and only if $\X_1*\Y_1 \cov \X_2*\Y_1$. 
%For suppose $\X_1 \cov \X_2$; then $\X_1*\Y_1 \prec \X_2*\Y_1$.
%Now if $\X_1*\Y_1 \prec \R\prec \X_2*\Y_1$ for some $\R\in\Sup(G)$,
%then $\R$ must be a $*$-product over $N$ by Lemma \ref{lemClosure}.
%So $\R= \S* \Y_1$ for some $\S \in \Sup(N)$,
%and then it follows that $\X_1 \prec \S \prec \X_2$, a contradiction.
%Therefore, $\X_1*\Y_1 \cov \X_2 * \Y_1$. 
%
%On the other hand, suppose $\X_1*\Y_1 \cov \X_2 * \Y_1$\ldots****
%By a similar argument it can be shown that $Y_1 \cov Y_2$ if and only if $\X_1*\Y_1 \cov \X_1*\Y_2$.
%
%Now if $\X_1=\X_2$ and $\Y_1 \cov \Y_2$,
%then $\X_1*\Y_1 \cov \X_1*\Y_2=\X_2*\Y_2$. 
%Likewise if $\Y_1=\Y_2$ and $\X_1 \cov \X_2$, 
%then $\X_1*Y_1 \cov \X_2*\Y_2$.
%
%So now suppose that $\X_1*\Y_1 \cov \X_2*\Y_2$. 
%Then either $\X_1 \prec \X_2$ or $\Y_1 \prec \Y_2$ or both by Lemma \ref{lemprec}.
%But if $\X_1\not=\X_2$ and $\Y_1\ne \Y_2$,
%then $\X_1*\Y_1 \prec \X_1*\Y_2 \prec \X_2*\Y_2$, a contradiction. 
%Therefore, either $\X_1=\X_2$ or $\Y_1=\Y_2$.
%If $\X_1=\X_2$, then $\Y_1 \cov \Y_2$ by our earlier claim, and we are done.
%Likewise if $\Y_1=\Y_2$, then $\X_1 \cov \X_2$, and we are done.
%\end{proof}

\def\AA{\mathcal{A}}
\def\BB{\mathcal{B}}
\begin{lemma}\label{lemusmlsm}
  Let $G$ be a group and let $N\norm G$. 
  Let $\AA = \Sup_G(N)$ and $\BB = \Sup(G/N)$. 
  Then the sublattice $\AA*\BB$ of $\Sup(G)$
  %$\zz=\{\xxx*\yyy :  \xxx \in \xx, \yyy\in \yy \} $ of $\Sup(G)$ 
  is upper semimodular if and only if $\AA$ and $\BB$ are both upper semimodular; 
  similarly, $\AA*\BB$ is lower semimodular if and only if $\AA$ and $\BB$ are both lower semimodular.
\end{lemma}

\begin{proof}
By \cite[Theorem 121]{donnellan}, we have that $\AA\otimes\BB$ is upper (or lower) semimodular 
if and only if $\AA$ and $\BB$ are both upper (or lower) semimodular. 
The result then follows from the lattice isomorphism of Lemma \ref{lemClosure}.
\end{proof}

\section{Direct product supercharacter theories of $\Cyc{pq}$}\label{sect_directproducts}

Throughout this section, 
let $p$ and $q$ be distinct primes,
let $G$ be cyclic of order $pq$,
and let $P$ and $Q$ be its subgroups of orders $p$ and $q$, respectively.
Since $G=P\idp Q$,
we may consider the subset
$$\Sup(P) \by \Sup(Q) = \{\bbb\by\ccc: \bbb\in\Sup(P), \ccc\in\Sup(Q)\}$$
of direct products in $\Sup(G)$.
If $\bbb\in\Sup(P)$,
   $\ccc\in\Sup(Q)$,
   $x\in P$, and $y\in Q$,
   then by definition
   $$[xy]_{\bbb\by\ccc} = [x]_\bbb [y]_\ccc.$$

Recall from \cite{ssr} that if $\xxx\in\Sup(G)$ and $N\leq G$ is $\xxx$-normal,
then the parts of $\xxx$ which lie in $N$ form a supercharacter theory of $N$
denoted $\xxx_N$.
Then since $[1]_\bbb = \{1\} = [1]_\ccc$,
both $P$ and $Q$ are $(\bbb\by\ccc)$-normal and
$(\bbb\by\ccc)_P=\bbb$ and $(\bbb\by\ccc)_Q = \ccc$.

Direct products can be quite useful.
For example, they provide upper bounds for supercharacter theories that come from automorphisms.

\begin{lemma}\label{lem:dpinequality}
  Suppose $\zzz\in\Sup(G)$ comes from automorphisms.
  Then $\zzz\leq\zzz_P\by\zzz_Q$.
\end{lemma}
\begin{proof}
  Since $\zzz$ comes from automorphisms, both $P$ and $Q$ are $\zzz$-normal,
  so $\zzz_P$ and $\zzz_Q$ exist.
  Now let $x\in P$ and $y\in Q$, so that $xy$ is an arbitrary element of $G$.
  Then consider
  $$\hat{[x]_\zzz} \, \hat{[y]_\zzz} = \hat{[x]_\zzz [y]_\zzz},$$
  which is a linear combination of superclass sums of $\zzz$ by Lemma \ref{superclassproducts}.
  Since one summand is $xy$,
  it follows that
  $$[xy]_\zzz \sseq [x]_\zzz [y]_\zzz = [x]_{\zzz_P} [y]_{\zzz_Q} = [xy]_{\zzz_P\by\zzz_Q}.$$
  Thus $\zzz\leq\zzz_P\by\zzz_Q$.
\end{proof}

In fact, the direct product supercharacter theories themselves come from automorphisms,
as the following lemma shows.
Recall that all supercharacter theories of $P$ and $Q$ come from automorphisms.

\begin{lemma}\label{lem:dpfromauts}
  Let $\bbb\in\Sup(P)$ and $\ccc\in\Sup(Q)$.
  Let $R\leq\Aut(P)$ and $S\leq\Aut(Q)$
  such that $\bbb$ comes from $R$ and $\ccc$ comes from $S$.
  Then $\bbb\by\ccc$ comes from a subgroup of $\Aut(G)$ isomorphic to $R\by S$.
\end{lemma}
\begin{proof}
%  Recall that since $G$ is cyclic of order $pq$,
%  its automorphism group is isomorphic to the group of units of $\mathbb{Z}/pq\mathbb{Z}$.
%  The restriction map $\Aut(G) \to \Aut(P)$
%  corresponds to working modulo $p$,
%  and likewise for $Q$.
%  Thus the Chinese Remainder Theorem
%  implies that
%  the map
%  $$\begin{array}{rcl}
%      \Aut(G) &\longrightarrow& \Aut(P)\by\Aut(Q) \\
%      \phi    &\longmapsto&     \left( \phi|_P, \phi|_Q \right)
%    \end{array}$$
%  is a group isomorphism.
%  Let $A$ be the preimage of $R\by S$ under this map,
%  so that $A\cong R\by S$,
%  and let $\ddd\in\Sup(G)$ come from $A$.
%  Then $\ddd\leq \ddd_P\by\ddd_Q = \bbb\by\ccc$ by Lemma \ref{lem:dpinequality}.
  
  Since $G=P\idp Q$ and both $P$ and $Q$ are characteristic subgroups of $G$,
  the map
  $$\begin{array}{rcl}
      \Aut(G) &\longrightarrow& \Aut(P)\by\Aut(Q) \\
      \phi    &\longmapsto&     \left( \phi|_P, \phi|_Q \right)
    \end{array}$$
  is a group isomorphism,
  and indeed the action of $\Aut(P)\by\Aut(Q)$ on $P\by Q$ is isomorphic to the action of $\Aut(G)$ on $G$.
  Consider the subgroup $R\by S$ of $\Aut(P)\by\Aut(Q)$,
  and let $A$ be the corresponding subgroup of $\Aut(G)$.
  For all $x\in P$ and $y\in Q$, 
  the orbit of $(x,y)\in P\by Q$ under the action of $R\by S$ is $x^R \by y^S$,
  so the orbit of $xy\in G$ under the action of $A$ is
  $$\{x'y': x'\in x^R, y'\in y^S\} = x^R y^S = [x]_\bbb [y]_\ccc = [xy]_{\bbb\by\ccc}.$$
  Thus $\bbb\by\ccc$ comes from $A\cong R\by S$, as desired.
%  
%  On the other hand,\marginpar{\tiny\raggedright shorter?  `The orbits of $R\by S$ on $P\by Q$ are precisely the superclasses of $\ddd$'}
%  let $g\in G$ and write $g=xy$ for some $x\in P$ and $y\in Q$.
%  Let $g'\in [g]_{\bbb\by\ccc}$;
%  then $g'=x'y'$ for some $x'\in P$ and $y'\in Q$.
%  Now by definition, $[g]_{\bbb\by\ccc} = [x]_\bbb [y]_\ccc$,
%  so $x'\in [x]_\bbb$ and $y'\in [y]_\ccc$.
%  Thus there exist $\alpha\in R$ and $\beta\in S$ such that $x'=x^\alpha$ and $y'=y^\beta$.
%  By the definition of $A$, then, there exists $\gamma \in A$ such that $\gamma|_P = \alpha$ and $\gamma|_Q=\beta$; hence
%  $$g^\gamma = (xy)^\gamma = x^\alpha y^\beta = x'y' = g'.$$
%  Thus $g'\in [g]_\ddd$, implying that $\bbb\by\ccc\leq\ddd$.
%  
%  It follows that $\bbb\by\ccc$ equals $\ddd$ and comes from $A$, as desired.
\end{proof}

We need one more piece of notation.
The map $\dot{} : x \mapsto Px$ is an isomorphism from $Q$ to $G/P$,
which induces a lattice isomorphism \mbox{$\dot{~} : \Sup(Q) \to \Sup(G/P)$.}
The definition of $*$-product implies
that if $y\in Q$, $\bbb\in\Sup(P)$, and $\ccc\in\Sup(Q)$, then
$[y]_{\bbb*\dot{\ccc}} = P[y]_\ccc$;
note too that $(\bbb*\dot{\ccc})_P=\bbb$.

\begin{lemma}\label{lem:dprestrict}
  Let $\ddd\in\Sup(G)$ come from automorphisms.
  If $\bbb\in\Sup(P)$ and $\ccc\in\Sup(Q)$ such that $\ddd\leq\bbb*\dot{\ccc}$,
  then $\ddd_P\leq \bbb$ and $\ddd_Q\leq \ccc$.
\end{lemma}
\begin{proof}
  Certainly $P$ and $Q$ are $\ddd$-normal since $\ddd$ comes from automorphisms,
  so $\ddd_P$ and $\ddd_Q$ exist;
  moreover, $\ddd_P \leq (\bbb*\dot{\ccc})_P = \bbb$.
  Now let $y\in Q$;
  then $[y]_\ddd \sseq [y]_{\bbb*\dot{\ccc}}$ and also $[y]_\ddd \sseq Q$,
  so 
  $$[y]_\ddd \sseq [y]_{\bbb*\dot{\ccc}}\cap Q = P[y]_\ccc \cap Q = [y]_\ccc.$$
  Thus $[y]_{\ddd_Q} = [y]_\ddd \sseq [y]_\ccc$,
  so $\ddd_Q \leq \ccc$.
\end{proof}

The following characterization of $\bbb\by\ccc$ will be used in 
Theorem \ref{lowersemimodularity}
to help prove that $\Sup(\Cyc{pq})$ is lower semimodular.
We shall use the fact that
%Note first that it is clear from the definition of the direct product that
%  if $\bbb_1,\bbb_2\in\Sup(P)$ and $\ccc_1,\ccc_2\in\Sup(Q)$,
%  then 
  $\bbb_1\by\ccc_1\leq \bbb_2\by\ccc_2$ if and only if both $\bbb_1\leq\bbb_2$ and $\ccc_1\leq\ccc_2$,
which is clear from the definition of the direct product.

\begin{lemma}\label{lem:dpunique}
  Let $G$ be cyclic of order $pq$, and let $P$ and $Q$ be the subgroups of orders $p$ and $q$, respectively.
  Let $\bbb\in\Sup(P)$ and $\ccc\in\Sup(Q)$.
  Then $\bbb\by\ccc\in\Sup(G)$ is the unique supercharacter theory coming from automorphisms that is covered by $\bbb*\dot{\ccc}$.
\end{lemma}
\begin{proof}
  Let $g\in G$ and write $g=xy$ for some $x\in P$ and $y\in Q$.
  If $g\in P$, then $[g]_{\bbb\by\ccc} = [g]_\bbb = [g]_{\bbb*\dot{\ccc}}$.
  If $g\not\in P$, however, then
  $$[g]_{\bbb\by\ccc} = [x]_\bbb[y]_\ccc \sseq P [y]_\ccc = [y]_{\bbb*\dot{\ccc}}.$$
  In either case, $[g]_{\bbb\by\ccc}$ is a subset of a superclass of $\bbb*\dot{\ccc}$,
  so $\bbb\by\ccc \leq \bbb*\dot{\ccc}$.
  
  To prove covering,
  suppose $\zzz\in\Sup(G)$ such that $\bbb\by\ccc < \zzz < \bbb*\dot{\ccc}$.
  Then $P$ is $\zzz$-normal and
  $(\bbb\by\ccc)_P \leq \zzz_P \leq (\bbb*\dot{\ccc})_P$, so $\bbb\leq\zzz_P\leq\bbb$ and thus $\zzz_P=\bbb$.
  Now by Theorem \ref{thm:classification}, either $\zzz$ comes from automorphisms or $\zzz$ is a $*$-product.
  In the former case, $\zzz_Q \leq \ccc$ by Lemma \ref{lem:dprestrict},
  but we also have $\ccc = (\bbb\by\ccc)_Q \leq \zzz_Q$,
  so $\zzz_Q=\ccc$.
  Then $\zzz\leq \zzz_P\by \zzz_Q = \bbb\by\ccc$, a contradiction.\marginpar{\tiny\raggedright Uses L\ref{lem:dpinequality}}
  
  If $\zzz$ is a $*$-product, on the other hand, 
  write $\zzz=\eee*\dot{\fff}$ 
  for some $\eee\in\Sup(P)$ and $\fff\in\Sup(Q)$,
  so that $\bbb\by\ccc \leq \eee*\dot{\fff} \leq \bbb*\dot{\ccc}$.
  Then $\dot{\fff}\leq\dot{\ccc}$ by Lemma \ref{lemClosure},
  so $\fff\leq\ccc$,
  while Lemma \ref{lem:dprestrict} implies that
  $\ccc=(\bbb\by\ccc)_Q \leq \fff$.\marginpar{\tiny\raggedright Uses L\ref{lem:dpfromauts}}
  Thus $\fff=\ccc$ and $\eee=\zzz_P=\bbb$,
  so $\zzz=\eee*\dot{\fff} = \bbb*\dot{\ccc}$,
  another contradiction.
  Hence $\bbb\by\ccc \cov \bbb*\dot{\ccc}$.

  Finally, to prove uniqueness suppose $\zzz\in\Sup(G)$ comes from automorphisms
  such that $\zzz\cov \bbb*\dot{\ccc}$.
  Then $\zzz_P\leq \bbb$ and $\zzz_Q\leq \ccc$ by Lemma \ref{lem:dprestrict},
  so $\zzz \leq \zzz_P\by\zzz_Q \leq \bbb\by\ccc\cov\bbb*\dot{\ccc}$ by Lemma \ref{lem:dpinequality}.
  Thus $\zzz=\bbb\by\ccc$, as desired.
\end{proof}

Note that the subset\marginpar{\tiny\raggedright \emph{Is} it a sublattice?}
$\Sup(P)\by\Sup(Q)$ is not in general convex in $\Sup(G)$.
For example, 
suppose $pq$ is odd and
let $\xinv{G}\in\Sup(\Cyc{pq})$ be the supercharacter theory 
coming from $\gen{\sigma}\leq\Aut(\Cyc{pq})$, where $\sigma$ is the inversion automorphism.
Then $\Min{P}\by \Min{Q} < \xinv{G} < \Max{P}\by\Max{Q}$,
but $\xinv{G}$ is not a direct product.
Nevertheless, the following lemma shows that covering relations in this subset do hold in the full lattice as well.
\begin{lemma}\label{lem:dpcovering}
  Let $G$ be cyclic of order $pq$, and let $P$ and $Q$ be the subgroups of orders $p$ and $q$, respectively.
  Let $\bbb_1,\bbb_2\in\Sup(P)$ and $\ccc_1,\ccc_2\in\Sup(Q)$.
  Suppose either that $\bbb_1=\bbb_2$ and $\ccc_1\cov\ccc_2$, 
  or else that $\bbb_1\cov\bbb_2$ and $\ccc_1=\ccc_2$.
  Then $\bbb_1\by\ccc_1\cov\bbb_2\by\ccc_2$ in the lattice $\Sup(G)$.
\end{lemma}
\begin{proof}
  We prove the case when $\bbb_1\cov\bbb_2$ and $\ccc_1=\ccc_2$;
  the other case follows by symmetry.
  Let $R_1,R_2\leq\Aut(P)$ and $S\leq\Aut(Q)$
  such that $\bbb_1$ comes from $R_1$, 
            $\bbb_2$ comes from $R_2$, 
            and $\ccc_1=\ccc_2$ comes from $S$.
  Now by Lemma \ref{lem:dpfromauts},
  for $i=1,2$ the direct product $\bbb_i\by\ccc_i$ comes from a subgroup of $\Aut(G)$ isomorphic to $R_i\by S_i$.
  By the lattice isomorphism of Lemma \ref{Aut(G)},
  we know that $R_1$ is a maximal subgroup of $R_2$.
  Since $\Aut(P)$ is abelian, this means that
  %$$\frac{|R_2\by S|}{|R_1\by S|} = \frac{|R_2|}{|R_1|}$$
  $|R_2:R_1|$
  is prime.
  %so $R_1\by S$ is a maximal subgroup of $R_2\by S$.

  Now suppose $\zzz\in\Sup(G)$ such that $\bbb_1\by\ccc_1 < \zzz < \bbb_2\by\ccc_2$.
  Then both $P$ and $Q$ are $\zzz$-normal, implying that $\zzz$ is not a $*$-product,
  so by Theorem \ref{thm:classification} it must come from automorphisms.
  By Lemmas \ref{lem:dpfromauts} and \ref{Aut(G)}, however,
  $\zzz$ would come from a group of automorphisms of an order strictly between $|R_1||S|$ and $|R_2||S|$,
  contradicting LaGrange's Theorem.
  Thus $\bbb_1\by\ccc_1 \cov \bbb_2\by\ccc_2$.
\end{proof}

\section{Upper Semimodularity}\label{sect_upper}

A lattice element is called an \defnstyle{atom} if it covers the minimal element of the lattice;
a lattice element covered by the maximal element of the lattice is likewise called a \defnstyle{coatom}.
%A lattice element is called a \defnstyle{coatom} if it is covered by the maximal element of the lattice;
%a lattice element covering the minimal element of the lattice is likewise called an \defnstyle{atom}.
Recall that if $\ccc$ is a supercharacter theory, $|\ccc|$ denotes how many superclasses it has.
The coatoms of some supercharacter theory lattices are easy to identify.

\begin{lemma}\label{lemMMCoatoms}
Let $G$ be a cyclic group of nonprime order $n$. 
Then $\xxx\in\Sup(G)$ is a coatom
if and only if 
$\xxx=\Max{N}*\Max{G/N}$
for some proper nontrivial subgroup $N\leq G$.
\end{lemma}

\begin{proof}
The supercharacter theory $\Max{N}*\Max{G/N}$ is clearly a coatom
for each $N\leq G$, since $|\Max{N}*\Max{G/N}|=3$ and $|\Max{G}|=2$.
Conversely,
let $\xxx\in\Sup(G)$ be a coatom.
By Theorem \ref{thm:classification},
either $\xxx$ is a direct product, $\xxx$ is a $\triangle$-product, or $\xxx$ comes from automorphisms.
In each case, there exists at least one proper nontrivial subgroup $N$ of $G$ that is $\xxx$-normal,
so $\xxx\pleq \Max{N}*\Max{G/N} \cov \Max{G}$;
since $\xxx$ is a coatom, it follows that $\xxx=\Max{N}*\Max{G/N}$, as desired.
\end{proof}

We next consider the meet of all these coatoms.
Recall that $\A{G}$ denotes the supercharacter theory of a group $G$
that comes from the full automorphism group $\Aut(G)$.

\begin{lemma}\label{coatommeets}
  Let $G$ be a cyclic group of nonprime order $n$. 
  Then 
  $$\A{G} = \bigwedge_{1<N<G} \Max{N}*\Max{G/N}.$$
\end{lemma}

\begin{proof}
  Because $G$ is cyclic,
  all elements of a given order lie in the same orbit under the action of $\Aut(G)$,
  so the superclasses of $\A{G}$ partition the elements of $G$ according to their orders.
  Thus $\A{G}\pleq \Max{N}*\Max{G/N}$ for each proper nontrivial subgroup $N$,
  so $\A{G}\pleq \bigwedge_{1<N<G} \Max{N}*\Max{G/N}$.

  Now consider two elements $x,y\in G$ of different orders $o(x)<o(y)$.
  Then $y\not\in\gen{x}$,
  so $x$ and $y$ lie in different superclasses of $\Max{\gen{x}}*\Max{G/{\gen{x}}}$;
  thus they lie in different superclasses of $\bigwedge_{1<N<G} \Max{N}*\Max{G/N}$.
  Hence every superclass of that meet contains elements of a single order,
  so $\bigwedge_{1<N<G}\Max{N}*\Max{G/N} \pleq \A{G}$.
\end{proof}

The goal of this section is to 
identify the positive integers $n$ for which the lattice $\Sup(\Cyc{n})$ is upper semimodular.
To that end, we address two specific cases: 
the cyclic groups of order $pq$ 
and the cyclic groups of order $p^2$.
For the $pq$ case 
we shall show that $\Min{P}*\Min{G/P}$ and $\Min{Q}*\Min{G/Q}$
violate the upper semimodularity criterion.

\begin{lemma}\label{lemcoprime}
  Let $G$ be a cyclic group of order $pq$, where $p$ and $q$ are distinct primes,
  and let $P$ and $Q$ be the subgroups of orders $p$ and $q$, respectively. 
  Let $\xxx\in \Sup(P)$, $\yyy\in \Sup(G/P)$, $\uuu\in \Sup(Q)$, and $\vvv\in \Sup(G/Q)$.
  Then $\xxx*\yyy$ is incomparable to $\uuu*\vvv$; moreover, $(\xxx*\yyy) \join (\uuu*\vvv) = \Max{G}$.
\end{lemma}

\begin{proof}
Some part of $\xxx*\yyy$ contains a nontrivial $P$-coset,
which contains both an element of order $q$ and elements of order $pq$.
Then because one part of $\Max{Q}*\Max{G/Q}$ comprises all the elements of order $q$,
it follows that $\xxx*\yyy\not\pleq \Max{Q}*\Max{G/Q}$;
likewise $\uuu*\vvv\not\pleq \Max{P}*\Max{G/P}$.
Therefore the join $(\xxx*\yyy)\join(\uuu*\vvv)$ is finer than neither coatom of $\Sup(G)$, 
\marginpar{\tiny\raggedright Uses Lemma \ref{lemMMCoatoms}}
so it can only be $\Max{G}$.
In particular, $\xxx*\yyy$ and $\uuu*\vvv$ are incomparable.
\end{proof}

\begin{lemma}\label{lemwrapup}
Let $G$ be a cyclic group of order $pq$, where $p$ and $q$ are distinct primes. 
Let $P$ and $Q$ be the subgroups of orders $p$ and $q$, respectively. 
Then $\Min{P}*\Min{G/P}$ and $\Min{Q}*\Min{G/Q}$ are atoms. 
Furthermore, neither $\Min{P}*\Min{G/P}$ nor $\Min{Q}*\Min{G/Q}$ is comparable with $\A{G}$.
\marginpar{\tiny\raggedright The last sentence \emph{is} needed later.}
\end{lemma}

\begin{proof}
By the definition of the direct product,
$\Min{G} = \Min{P}\by\Min{Q}$,
so $\Min{P}*\Min{G/P}$ is an atom
by Lemma \ref{lem:dpunique}.
Now some superclass of $\Min{P}*\Min{G/P}$ contains a nontrivial $P$-coset, 
and hence includes elements both of order $q$ and of order $pq$, 
whereas all the elements in each superclass of $\A{G}$ have the same order.
Thus $\Min{P}*\Min{G/P} \not \pleq \A{G}$.
On the other hand, $\A{G}\not\pleq \Min{P}*\Min{G/P}$ since the latter is an atom, so $\Min{P}*\Min{G/P}$ and $\A{G}$ are incomparable.
By symmetry, $\Min{Q}*\Min{G/Q}$ is also an atom and incomparable with $\A{G}$.
\end{proof}

We pause to note a corollary that will be useful in Section \ref{sect_lower}.
In the lattice $\Sup(G)$, if $\xxx\leq\yyy$ let us say that $\yyy$ \defnstyle{lies above} $\xxx$;
likewise let us say $\yyy$ \defnstyle{lies between} $\xxx$ and $\zzz$ if $\xxx\leq\yyy\leq\zzz$.

\begin{corollary} \label{corwrapup}
  Let $G$ be a cyclic group of order $pq$, where $p$ and $q$ are distinct primes. 
  %Let $P$ and $Q$ be the subgroups of order $p$ and $q$ respectively. 
  Then in $\Sup(G)$,
  \renewcommand\theenumi{\alph{enumi}}
  \begin{enumerate}
    \item No supercharacter theory that comes from automorphisms lies above a nontrivial $*$-product. \label{noautabovestar}
    \item No nontrivial $*$-product comes from automorphisms. \label{noautisstar}
    \item No nontrivial $*$-product lies between two supercharacter theories from automorphisms, 
          and no supercharacter theory from automorphisms lies between two nontrivial $*$-products.\label{nostarbetween}
  \end{enumerate}
\end{corollary}
\begin{proof}
  Let $\yyy\in\Sup(G)$ come from automorphisms
  and let $\zzz\in\Sup(G)$ be a $*$-product over, without loss of generality, $P$.
  If $\zzz \leq \yyy$, then $\Min{P}*\Min{G/P} \leq \zzz \leq \yyy \leq \A{G}$, contradicting Lemma \ref{lemwrapup};
  this proves part (\ref{noautabovestar}).
  Parts (\ref{noautisstar}) and (\ref{nostarbetween}) follow immediately.
\end{proof}

We now resolve the upper semimodularity of $\Sup(\Cyc{pq})$.

\begin{lemma}\label{lempqusm}
  Let $G$ be a cyclic group of order $pq$, where $p$ and $q$ are distinct primes. 
  Then $\Sup(G)$ is not upper semimodular.
\end{lemma}
\begin{proof}
  Let $P$ and $Q$ be the subgroups of orders $p$ and $q$ respectively;
  without loss of generality suppose $p>2$.
  We know by Lemma \ref{lemwrapup} that $\Min{P}*\Min{G/P}$ and $\Min{Q}*\Min{G/Q}$ are atoms 
  and therefore cover their meet. 
  Furthermore, $(\Min{P}*\Min{G/P}) \join (\Min{Q}*\Min{G/Q}) = \Max{G}$ by Lemma \ref{lemcoprime}, 
  but $\Min{P}*\Min{G/P} < \Max{P}*\Max{G/P} < \Max{G}$. 
  Therefore 
  %$\Min{P}*\Min{G/P} \not \cov \Min{P}*\Min{G/P} \join \Min{Q}*\Min{G/Q}$ and hence 
  $\Sup(G)$ is not upper semimodular.
\end{proof}

We next consider cyclic groups of order $p^2$, with $p$ an odd prime;
to do so, we need a new atom.
For every abelian group $G$,
let $\xinv{G}$ denote the supercharacter theory of $G$
that comes from $\gen{\sigma}$,
where $\sigma$ is the inversion automorphism $x\mapsto x^{-1}$.

We shall prove that $\Sup(\Cyc{p^2})$ is neither upper nor lower semimodular 
by showing that it contains the sublattice depicted in
Figure \ref{figMinInv}.
The next four lemmas will verify the coverings portrayed in the figure.
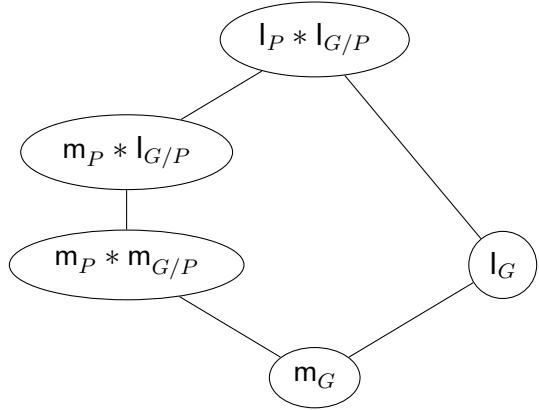
\begin{figure}[htbp]
\begin{center}
\begin{tikzpicture}[xscale=2.5, yscale=1.5]
  \tikzstyle{every node}=[ellipse,draw]
  \node (II) at ( 0,3) {$\xinv{P}*\xinv{G/P}$};
  \node (mI) at (-1,2) {$\Min{P}*\xinv{G/P}$};
  %\node (Im) at ( 0,2) {$\xinv{P}*\Min{G/P}$};
  \node (mm) at (-1,1) {$\Min{P}*\Min{G/P}$};
  \node (I)  at ( 1,1) {$\xinv{G}$};
  \node (m)  at ( 0,0) {$\Min{G}$};
  \draw (m) -- (I) -- (II);
  \draw (m) -- (mm) -- (mI) -- (II);
  %\draw (mm) -- (Im) -- (II);
\end{tikzpicture}
\caption{Hasse diagram of a sublattice of $\Sup(\Cyc{p^2})$, where $p$ is an odd prime
}
\label{figMinInv}
\end{center}
\end{figure}

\begin{lemma}\label{leminv}
 Let $G$ be a cyclic group of order greater than 2. 
 Then $\xinv{G}$ is an atom in $\Sup(G)$.
\end{lemma}

\begin{proof}

Let $\yyy < \xinv{G}$; we wish to show that $\yyy=\Min{G}$. 
Let $g$ be a generator of $G$ and let $n=|G|$.
Then the superclasses of $\xinv{G}$ are of the form $\{g^k,g^{-k}\}$ for integers $k$;
since $\yyy<\xinv{G}$, 
there exists at least one integer $k$ such that
the superclass $\{g^k, g^{-k}\}$ of $\xinv{G}$ is broken up into singletons $\{g^k\}$ and $\{g^{-k}\}$ in $\yyy$;
note that $g^{2k} \neq 1$. 
We shall show that $\{g\}$ is a superclass of $\yyy$. 
If $g^k=g$, we are done. 
Otherwise consider the superclass of $\yyy$ containing $g^{1-k}$. 
This superclass is either $ \{g^{1-k},g^{k-1}\}$ or $\{g^{1-k}\}$. 
In the former case,
\begin{equation}\label{invatomeqn1}
  \left(g^k\right) \left(g^{1-k} + g^{k-1}\right) = g +g^{2k-1}
\end{equation}
must be a linear combination of sums of superclasses of $\yyy$.
Now $g^{2k}\neq 1$ so $g^{2k-1}\neq g^{-1}$;
since $g$ appears in (\ref{invatomeqn1}) without $g^{-1}$,
we see that $\{g\}$ must be a superclass of $\yyy$. 
If the latter case holds, 
so that $g^{1-k}$ is a singleton in $\yyy$, 
then 
$$(g^k) (g^{1-k}) = g$$
is a linear combination of sums of superclasses of $\yyy$,
so again $\{g\}$ must be a superclass of $\yyy$.

Now since the generator $g$ belongs to a singleton superclass of $\yyy$,
it follows that $g^{m}$ is a linear combination of superclass sums of $\yyy$
for all $m\in\N$. 
Hence $\yyy=\Min{G}$, so $\xinv{G}$ is an atom.
\end{proof}

\begin{lemma}\label{mpmp}
Let $G$ be a cyclic group of order $p^2$, where $p$ is an odd prime. 
Let $P$ be the subgroup of order $p$. 
Then $\Min{P}*\Min{G/P}$ is an atom in $\Sup(G)$.
\end{lemma}

\begin{proof}
Suppose $\zzz<\Min{P}*\Min{G/P}$ for some $\zzz\in\Sup(G)$. 
Since no $*$-product lies below $\Min{P}*\Min{G/P}$,
by Theorem \ref{thm:classification}
we know that $\zzz$ comes from some group of automorphisms $H\leq\Aut(G)$.
However, 
%$$\Min{P}*\Min{G/P}=\{1/g^p/g^{2p}/.../g^{p^2-1}/Pg/Pg^2/.../Pg^{p-1}\}$$
$\Min{P}*\Min{G/P}$
also comes from a group of automorphisms $K$ of order $p$,
generated by the map $x\mapsto x^{p+1}$.
Then $H$ must be a proper subgroup of $K$ by Lemma \ref{Aut(G)},
so $H=1$.
Therefore $\zzz=\Min{G}$, so $\Min{P}*\Min{G/P}$ is an atom.
\end{proof}

\begin{lemma}\label{leminvminmin}
Let $G$ be a cyclic group of order $p^2$, where $p$ is an odd prime. 
Let $P$ be the subgroup of order $p$. 
Then $\left(\Min{P} * \Min{G/P}\right) \join \xinv{G} = \xinv{P}*\xinv{G/P}$.
\end{lemma}

\begin{proof}
Let $g$ be a generator of $G$.
We begin by noting that the superclasses of $\Min{P}*\Min{G/P}$ are
$$\{\{x\}: x\in P\} \cup \{Px: x\in G-P\}$$
%$$\{\{g^{kp}\} : k\in\mathbb{Z}\} \cup \{Pg^\ell: \ell\in\{1,\ldots,p-1\}\}$$
while the superclasses of $\xinv{G}$ are 
$\{\{x,x^{-1}\}:x\in G\}$.
%$\{\{g^k,g^{-k}\}:k\in\mathbb{Z}\}$.
Let $\yyy=\left(\Min{P}*\Min{G/P}\right) \join \xinv{G}$,
and let $x\in G$.
If $x\in P$, then the superclass of $x$ in $\yyy$ is simply $\{x,x^{-1}\}$.
If $x\not\in P$, then the superclass of $x$ in $\yyy$ is $Px \cup Px^{-1}$.
Therefore
$\yyy=\xinv{P}*\xinv{G/P}$, as desired.
\end{proof}

%THIS LEMMA IS NEEDED TO COMPLETE THE HASSE DIAGRAM,
%SO WE CAN DISPROVE LSM.
\begin{lemma}\label{inverses}
Let $G$ be a cyclic group of order $p^2$, where $p$ is an odd prime. 
Let $P$ be the subgroup of order $p$. 
Then $\xinv{G}$ is covered by $\xinv{P}*\xinv{G/P}$.\marginpar{\tiny\raggedright Needed to prove not LSM}
\end{lemma}

\begin{proof}
Certainly $\xinv{G} < \xinv{P}*\xinv{G/P}$.
Now suppose for a contradiction that $\xxx\in\Sup(G)$ such that $\xinv{G} < \xxx < \xinv{P}*\xinv{G/P}$.
By Theorem \ref{thm:classification}, 
   either $\xxx$ is a nontrivial $*$-product or 
   $\xxx$ comes from automorphisms (or both).
If $\xxx$ is a $*$-product, then $\Min{P}*\Min{G/P} \pleq \xxx$,
so $(\Min{P}*\Min{G/P})\join\xinv{G} \leq \xxx$, contradicting Lemma \ref{leminvminmin}.

Thus $\xxx$ comes from automorphisms.
Now $\xinv{P}*\xinv{G/P}$ also comes from automorphisms,
namely from a subgroup $B\leq \Aut(G)$ of order $2p$,
while $\xinv{G}$ comes from a subgroup $A\leq\Aut(G)$ of order $2$.
Then the lattice isomorphism of Lemma \ref{Aut(G)} requires $\xxx$ to come from a subgroup $C\leq\Aut(G)$
with $A<C<B$, contradicting Lagrange's theorem.
\end{proof}

We are now ready to resolve the semimodularity question for $\Sup(\Cyc{p^2})$.

\begin{lemma}\label{p^2}
Let $G$ be a cyclic group of order $p^2$, where $p$ is an odd prime. 
Then $\Sup(G)$ is neither upper semimodular nor lower semimodular.
\end{lemma}

\begin{proof}
Lemmas \ref{leminv} through \ref{inverses}
have verified all but two of the coverings in the Hasse diagram of Figure \ref{figMinInv}.
That $\Min{P}*\Min{G/P}\cov \Min{P}*\xinv{G/P}$
and that $\Min{P}*\xinv{G/P}\cov \xinv{P}*\xinv{G/P}$
follow from Lemmas \ref{leminv} and \ref{lemClosure}.
Then
$\Min{P}*\Min{G/P}$ and $\xinv{G}$ cover their meet but are not both covered by their join,
so $\Sup(G)$ is not upper semimodular.
Likewise
$\Min{P}*\xinv{G/P}$ and $\xinv{G}$ are covered by their join
but do not both cover their meet, so $\Sup(G)$ is not lower semimodular.
\end{proof}

Corollary \ref{inheritance} now allows us to use
Lemmas \ref{lempqusm} and \ref{p^2} 
to prove a necessary and sufficient condition for upper semimodularity
of the supercharacter theory lattices of cyclic groups.

\begin{theorem}\label{upper semimodular}
Let $G$ be a cyclic group. Then $\Sup(G)$ is upper semimodular if and only if the order of $G$ is prime or four.
\end{theorem}

\begin{proof}
If $|G|$ is prime then $\Sup(G)$ is upper semimodular by Corollary \ref{Cp_modular}.
If $|G|=4$, we compute that $\Sup(G)$ is a three-element chain and so is upper semimodular.

Now suppose the order of $G$ equals $n$ where $n$ is neither prime nor four. 
Then $n$ is either a multiple of two distinct primes, 
a multiple of a square of an odd prime, 
or a multiple of 8.
In each case we may apply Corollary \ref{inheritance}.
If $n$ is a multiple of $pq$, then $\Sup(G)$ is not upper semimodular by Lemma \ref{lempqusm}. 
If $n$ is a multiple of $p^2$ where $p$ is odd, then Lemma \ref{p^2} implies that $\Sup(G)$ is not upper semimodular. 
Finally, we can compute that $\Sup(\Cyc{8})$ is not upper semimodular,
so if $n$ is a multiple of 8, then $\Sup(G)$ is not upper semimodular.
\end{proof}

\section{Lower Semimodularity}\label{sect_lower}

In order to obtain necessary and sufficient conditions for lower semimodularity,
we must look at two specific cases:  cyclic groups of orders $pq$ and $pqr$.

\begin{theorem}\label{lowersemimodularity}
  Let $G$ be a cyclic group of order $pq$, where $p$ and $q$ are distinct primes. 
  Then $\Sup(G)$ is lower semimodular.
\end{theorem}
\begin{proof}
  Let $P$ and $Q$ be the subgroups of orders $p$ and $q$ respectively. 
  Suppose $\xxx,\yyy\in\Sup(G)$ are covered by their join,
  and let $\zzz=\xxx \join \yyy$. 
  We shall show that $\xxx$ and $\yyy$ cover $\xxx\meet\yyy$.
  
  Since $G$ is cyclic of order $pq$, 
  by Theorem \ref{thm:classification} 
  there are only three possibilities for $\zzz$: 
  it could be the maximal supercharacter theory, a supercharacter theory coming from automorphisms, or a $*$-product.

  First, suppose $\zzz=\Max{G}$. 
  Then $\xxx$ and $\yyy$ are distinct coatoms,
  namely $\Max{P}*\Max{G/P}$ and $\Max{Q}*\Max{G/Q}$ by Lemma \ref{lemMMCoatoms},
  and hence by Lemma \ref{coatommeets}, 
  $\xxx \meet \yyy$ is the supercharacter theory that comes from $\Aut(G)$,
  whose superclasses partition $G$'s elements by their orders.
  Since there are four divisors of $pq$, 
  we have $|\xxx \meet \yyy|=4$
  while $|\xxx|=|\yyy|=3$.
  Therefore $\xxx$ and $\yyy$ cover $\xxx \meet \yyy$, as desired.

  Now suppose $\zzz$ comes from automorphisms. 
  Then $\xxx$ and $\yyy$ cannot be $*$-products by Corollary \ref{corwrapup},
  so they both must come from automorphisms. 
  Consider the subset $L$ of $\Sup(G)$ consisting of all supercharacter theories coming from automorphisms;
  by Corollary \ref{corwrapup}, it must be convex,
  and it contains its least upper bound $\A{G}$ and greatest lower bound $\Min{G}$;
  thus it is a convex sublattice.
  %no $*$-products can be between two supercharacter theories that come from automorphisms in $\Sup(G)$. 
  %Furthermore, the maximal supercharacter cannot be between two supercharacter theories that come from automorphisms. 
  %Therefore, covering relations in the sublattice of supercharacter theories from automorphisms remain the same as in the full lattice, $\Sup(G)$. 
  By Lemma \ref{Aut(G)}, $L$ is isomorphic to the subgroup lattice of $\Aut(G)$.
  Since the subgroup lattice of $\Aut(G)$ is modular because $\Aut(G)$ is abelian,
  it follows that $L$ is modular and hence lower semimodular.
  Thus $\xxx$ and $\yyy$ cover $\xxx \meet \yyy$ in $L$,
  and since $L$ is convex in $\Sup(G)$,
  they must cover their meet in $\Sup(G)$ as well.

  It remains to consider the case that $\zzz$ is a $*$-product; 
  without loss of generality, take it to be a $*$-product over $P$. 
  If $\xxx$ and $\yyy$ both came from automorphisms, 
  then so would $\zzz=\xxx \join \yyy\leq\A{G}$ by the convexity of the sublattice $L$, 
  contradicting Corollary \ref{corwrapup}. \marginpar{\tiny\raggedright Uses Cor.~\ref{corwrapup}}
  Thus without loss of generality, $\yyy$ is a $*$-product; since $\yyy<\zzz$, 
  it must be a $*$-product over $P$ rather than over $Q$ by Lemma \ref{lemcoprime}.
  There are now two subcases to consider.
    First, if $\xxx$ is also a $*$-product,
    then it is a $*$-product over $P$ by the same reason as before.
    Now the sublattice of $*$-products over $P$ is convex by Lemma \ref{lemClosure}, 
    %Therefore, covering relations in the $*$-product sublattice remain the same in the full lattice, $\Sup(G)$. 
    and it is lower semimodular by Lemma \ref{lemusmlsm} 
    because $\Sup(P)$ and $\Sup(G/P)$ are lower semimodular
    by Corollary \ref{Cp_modular}. 
    \marginpar{\tiny\raggedright Uses Lemma \ref{lemusmlsm}}
    Therefore $\xxx$ and $\yyy$ cover $\xxx \meet \yyy$.
    
    The remaining possibility is that $\xxx$ comes from automorphisms. 
    Since $\yyy$ and $\zzz$ are two $*$-products over $P$ with $\yyy\cov\zzz$,
    by Lemma \ref{lemcover} we can write $\yyy=\rrr_1*\dot{\sss_1}$ and $\zzz=\rrr_2*\dot{\sss_2}$ 
    for some $\rrr_1,\rrr_2\in\Sup(P)$ and $\sss_1,\sss_2\in\Sup(Q)$ such that 
    either $\rrr_1 \cov \rrr_2$ and $\sss_1=\sss_2$, or else $\rrr_1=\rrr_2$ and $\sss_1\cov\sss_2$. 
    Then by Lemma \ref{lem:dpunique},
    we have that $\xxx=\rrr_2\by\sss_2$. 
    Let $\www=\rrr_1\by\sss_1$;
    then by Lemma \ref{lem:dpunique},\marginpar{\tiny\raggedright Uses Lemma \ref{lem:dpunique}}
    we know that $\www \cov \rrr_1*\dot{\sss_1}$.
    Moreover,
    by Lemma \ref{lem:dpcovering}, we also know that $\www \cov \xxx$. \marginpar{\tiny\raggedright Uses Lemma \ref{lem:dpcovering}}
    Thus, $\www=\xxx\meet\yyy$ and so both $\xxx$ and $\yyy$ cover $\xxx \meet \yyy$,
    as shown in Figure \ref{figwxyz}.
    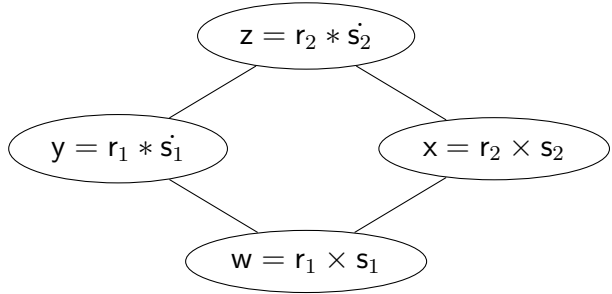
\begin{figure}[htbp]
    \begin{center}
    \begin{tikzpicture}[xscale=2.5, yscale=1.5]
      \tikzstyle{every node}=[ellipse,draw]
      \node (z) at ( 0,1) {$\zzz = \rrr_2*\dot{\sss_2}$};
      \node (y) at (-1,0) {$\yyy = \rrr_1*\dot{\sss_1}$};
      \node (x) at (1,0) {$\xxx=\rrr_2\by\sss_2$};
      \node (w)  at (0,-1) {$\www=\rrr_1\by\sss_1$};
      \draw (w) -- (x) -- (z);
      \draw (w) -- (y) -- (z);
    \end{tikzpicture}
    \caption{Hasse diagram of a sublattice of $\Sup(\Cyc{pq})$}
    \label{figwxyz}
    \end{center}
    \end{figure}
  
  Since $\xxx$ and $\yyy$ cover their meet in all possible cases for $\zzz$, 
  we conclude that $\Sup(G)$ is lower semimodular.
\end{proof}

The preceding theorem showed that $\Sup(\Cyc{pq})$ is lower semimodular.
If a cyclic group's order has a third prime factor, however,
then its supercharacter theory lattice is not lower semimodular.

\begin{lemma}\label{pqr}
  Let $G$ be a cyclic group of order $pqr$,
  where $p$, $q$, and $r$ are primes and $p\neq q$. 
  Then $\Sup(G)$ is not lower semimodular.
\end{lemma}

\begin{proof}
Let $P$, $Q$, and $N$ be the subgroups of order $p$, $q$, and $pq$ respectively. 
Then $\Max{P}*\Max{G/P}$ and $\Max{Q}*\Max{G/Q}$ are coatoms by Lemma \ref{lemMMCoatoms}, and hence are covered by their join.
We claim that 
the meet of $\Max{P}*\Max{G/P}$ and $\Max{Q}*\Max{G/Q}$ is $\A{N}*\Max{G/N}$.
Indeed, consider the superclasses of those three supercharacter theories.
  \begin{eqnarray*}
  \begin{array}{ll}
  \mbox{supercharacter} \\
  \mbox{theory}            & \mbox{superclasses} \\ \hline
  \Max{P}*\Max{G/P}        & \big\{\{1\},~\{\mbox{elements of order $p$}\},~\{\mbox{all other elements}\}\big\} \\
  \Max{Q}*\Max{G/Q}        & \big\{\{1\},~\{\mbox{elements of order $q$}\},~\{\mbox{all other elements}\}\big\} \\
  \A{N}*\Max{G/N}          & \big\{\{1\},~\{\mbox{elements of order $p$}\},~\{\mbox{elements of order $q$}\}, \\
                           &      \qquad \{\mbox{elements of order $pq$}\},~\{\mbox{all other elements}\}\big\}
  \end{array}
  \end{eqnarray*}
  Clearly $\A{N}*\Max{G/N}$ is a lower bound for $\Max{P}*\Max{G/P}$ and $\Max{Q}*\Max{G/Q}$.
  If it is not itself the meet of $(\Max{P}*\Max{G/P})$ and $(\Max{Q}*\Max{G/Q})$,
  then that meet would have to have superclasses
  % Then $\AAA*\Max{G/N} \pleq \Y$ since $\AAA*\Max{G/N}$ is a lower bound. Also, $\Y \prec \Max{P}*\Max{G/P},\Max{Q}*\Max{G/Q}$ since $\Max{P}*\Max{G/P}$ and $\Max{Q}*\Max{G/Q}$ are incomparable. Suppose that $\AAA*\Max{G/N} \prec \Y$. Then $\Y$ must have four parts as $\AAA*\Max{G/N}$ has five parts, and $\Max{P}*\Max{G/P}$ and  $\Max{Q}*\Max{G/Q}$ both have three parts. The only way to form $\Y$ such that the conditions above are satisfied is to join the part containing elements of order $pq$ with the part containing $N$ cosets in $\AAA*\Max{G/N}$.
  \begin{equation}\label{eqn:nontheory}
    \big\{\{1\},~\{\mbox{elements of order $p$}\},~\{\mbox{elements of order $q$}\},~\{\mbox{all other elements}\}\big\}.
  \end{equation}
  Now multiplying the sum of the elements of order $p$ with 
                  the sum of the elements of order $q$ 
  would yield a sum of elements of order $pq$,
  which is not a linear combination of sums of parts of (\ref{eqn:nontheory});
  hence (\ref{eqn:nontheory}) does not correspond to a supercharacter theory of $G$.
  Therefore $(\Max{P}*\Max{G/P}) \meet (\Max{Q}*\Max{G/Q})=\A{N}*\Max{G/N}$,
  as claimed.
This is not covered by $\Max{P}*\Max{G/P}$, however, since
$$ %(\Max{P}*\Max{G/P}) \meet (\Max{Q}*\Max{G/Q})
   %= 
   \A{N}*\Max{G/N}
   < (\Max{P}*\Max{N/P})*\Max{G/N}
   < \Max{P}*\Max{G/P},$$
so we conclude that $\Sup(G)$ is not lower semimodular.
%We will show that $\Max{P}*\Max{G/P} \meet \Max{Q}*\Max{G/Q} \ncov \Max{P}*\Max{G/P}$.
%
%We know that $\Max{P}*\Max{G/P} \meet \Max{Q}*\Max{G/Q}=\AAA*\Max{G/N}$ where $\AAA$ is the supercharacter theory resulting from $\Aut(N)$ by Lemma \ref{maxmeets}. Let $\Y$ be the partition between $\AAA*\Max{G/N}$ and $\Max{P}*\Max{G/P}$ as defined below.
%
%Note:
%$$\Max{P}*\Max{G/P}=\{1~/~g^{qr},g^{2qr},...,g^{(p-1)qr}~/~Pg \cup Pg^2 \cup ... \cup Pg^{qr-1}\}$$
%$$\Y=\{1~/~g^{qr},...,g^{(p-1)qr}~/~g^{pr},g^{2pr},...,g^{(q-1)pr}~g^r,...,g^{(pq-1)r}~/~Ng \cup ... \cup Ng^{r-1}\}$$
%$$\AAA*\Max{G/N}=\{1~/~g^{qr},g^{2qr}...,g^{(p-1)qr}~/~g^{pr},g^{2pr},...,g^{(q-1)pr}~/~g^r,g^{2r}...,g^{(pq-1)r}~/~Ng \cup Ng^2 \cup ... \cup Ng^{r-1}\}$$
%
%We know that $\Y$ is a supercharacter theory as $\Y=(\Max{P}*\Max{N/P})*\Max{G/P}$. 
%Therefore $\AAA*\Max{G/N} \ncov \Max{P}*\Max{G/P}$, 
%so $\Max{P}*\Max{G/P} \meet \Max{Q}*\Max{G/Q} \ncov \Max{P}*\Max{G/P}$ and hence, $\Sup(G)$ is not lower semimodular.
\end{proof}

\begin{theorem}\label{LSM}
Let $G$ be a cyclic group. Then $\Sup(G)$ is lower semimodular if and only if the order of $G$ is prime, the product of two distinct primes, or four.
\end{theorem}

\begin{proof}
  If $|G|$ is prime, 
    then $\Sup(G)$ is lower semimodular by Corollary \ref{Cp_modular}. 
  If $|G|$ is the product of two distinct primes, 
    then $\Sup(G)$ is lower semimodular by Lemma \ref{lowersemimodularity}.
  If $|G|=4$, we compute that $\Sup(G)$ is lower semimodular. 

  Suppose the order of $G$ is neither prime, nor the product of two distinct primes, nor four,
  and consider the number of distinct prime factors of $|G|$. 
%  In each case we apply Corollary \ref{inheritance}.
  First suppose $|G|=p^a$. 
    If $p$ is even, then $8$ divides $|G|$; we calculate that $\Sup(\Cyc{8})$ is not lower semimodular,
    so by Corollary \ref{inheritance}, neither is $\Sup(G)$.
    If $p$ is odd, then $p^2$ divides $|G|$ and $\Sup(G)$ is not lower semimodular by Lemma \ref{p^2}. 
  Next suppose $|G|=p^a q^b$. 
    Without loss of generality we assume $a\geq 2$;
    then $p^2 q$ divides $|G|$ and $\Sup(G)$ is not lower semimodular by Lemma \ref{pqr}.
  Finally, if at least three distinct primes divide $|G|$, say $p$, $q$, and $r$, 
    then $pqr$ divides $|G|$, and $\Sup(G)$ is not lower semimodular by Lemma \ref{pqr}.
\end{proof}

%%%%%%%%
% bib
%%%%%%%%

%%  \begin{thebibliography}{9}
%%
%%
%%  \bibitem{jca} J. Abbott, \emph{Sets Lattices and Boolean Algebras}, Boston, Allyn and Bacon (1969).  
%%  %just a book about sets, lattices, and stuff.  we might use for definitions.
%%  %
%%  \bibitem{diaconis_isaacs} P. Diaconis, I. M. Isaacs, ``Supercharacters and Superclasses for Algebra Groups", \emph{Trans. Amer. Math. Soc.} 360 (2008) 2359--2392.
%%  %
%%  \bibitem{ah1} Anders O. F. Hendrickson, ``Construction of Supercharacter Theories of Finite Groups". %what else goes here?
%%  \bibitem{ssr} Anders O. F. Hendrickson, ``Supercharacter theory constructions corresponding to Schur ring products." %what else goes here?
%%
%%  \bibitem{leungman} Leung, K. H., Man, S. H. (1996). On Schur rings over cyclic groups, II. \emph{J. Algebra} 183:273--285.
%%  \bibitem{donnellan} Thomas Donnellan, \emph{Lattice Theory}, London, Pergamon Press Ltd. (1968).
%%
%%
%%  \end{thebibliography}

\bibliography{semimodularity_bib}{}
\bibliographystyle{commalg}

\end{document}